\documentclass[12pt]{amsart}

\usepackage{amsmath,amsthm,amsfonts,amssymb}
\usepackage[all]{xy}

\numberwithin{equation}{section}

\newtheorem{thm}[equation]{Theorem}
\newtheorem{prop}[equation]{Proposition}
\newtheorem{cor}[equation]{Corollary}
\newtheorem{lem}[equation]{Lemma}
\theoremstyle{definition}

\newtheorem{remark}[equation]{Remark}

\DeclareMathOperator{\cha}{char}

\DeclareMathOperator{\ev}{ev}
\DeclareMathOperator{\cx}{cx}
\DeclareMathOperator{\coh}{H}
\DeclareMathOperator{\Ext}{Ext}
\DeclareMathOperator{\Hom}{Hom}

\DeclareMathOperator{\Irr}{Irr}

\DeclareMathOperator{\rank}{rank}

\newcommand{\Z}{\mathbb{Z}}
\newcommand{\C}{\mathbb{C}}
\newcommand{\ot}{\otimes}
\newcommand{\V}{\mathcal{V}}
\newcommand{\g}{\mathfrak{g}}
\newcommand{\m}{\mathfrak{m}}
\newcommand{\DOT}{\setlength{\unitlength}{1pt}\begin{picture}(2.5,2)
                  (1,1)\put(2,3.5){\circle*{2}}\end{picture}}

\begin{document}

\title[Support varieties and representation type]
{Support varieties and representation type of small quantum groups}

\author{J\"org Feldvoss}
\address{Department of Mathematics and Statistics,
University of South Alabama, Mobile, AL 36688--0002, USA}
\email{jfeldvoss@jaguar1.usouthal.edu}

\author{Sarah Witherspoon}
\address{Department of Mathematics, Texas A\&M University,
College Station, TX 77843--3368, USA}
\email{sjw@math.tamu.edu}

\date{November 9, 2009}

\maketitle


\begin{abstract}
In this paper we provide a wildness criterion for any finite dimensional
Hopf algebra with finitely generated cohomology. This generalizes a result
of Farnsteiner to not necessarily cocommutative Hopf algebras over ground
fields of arbitrary characteristic. Our proof uses the theory of support
varieties for modules, one of the crucial ingredients being a tensor product
property for some special modules. As an application we prove a conjecture
of Cibils stating that small quantum groups of rank at least two are wild.
\end{abstract}


\section{Introduction}

Rickard \cite{R} discovered an important connection between the representation
type of a finite dimensional self-injective algebra over an algebraically closed
field and the complexity of its modules: If there is a module with complexity
greater than two, then the algebra is expected to be wild. However there is a
gap in the proof. Rickard relies on \cite[Lemma 1]{R}, for which there is a
counterexample that was found recently, namely the truncated polynomial algebra
$k[x,y]/(x^2,y^2)$ (see \cite{F}). Farnsteiner \cite{F} later recovered Rickard's
result for finite group schemes, using the theory of support varieties. Bergh
and Solberg \cite{BS} adapted Farnsteiner's approach to prove Rickard's result 
for self-injective algebras under some finiteness conditions on Ext-algebras of
modules. 

In this paper, we show that Farnsteiner's geometric methods apply more directly
to any finite dimensional Hopf algebra whose cohomology is finitely generated as
an algebra over an algebraically closed field of arbitrary characteristic. (More
precisely, see assumption ({\bf fg}) below.) 
Under this assumption, one may define the
cohomological support variety of any finite dimensional module. Many of the
expected properties hold just as for finite groups and, more generally, finite
group schemes (i.e., finite dimensional cocommutative Hopf algebras). Some
important classes of non-cocommutative finite dimensional Hopf algebras are known
to satisfy our finite generation assumption ({\bf fg}) under some restrictions
on the parameters: Lusztig's small quantum groups
\cite{BNPP,GK} and, more generally, finite dimensional pointed Hopf algebras
having abelian groups of group-like elements \cite{MPSW}, or certain truncated
quantized function algebras of simply-connected complex semisimple 
algebraic groups \cite{G1}. 
Moreover, Etingof and
Ostrik \cite{EO} conjectured that the cohomology of any finite tensor category
is finitely generated; in case the tensor category is the representation category
of a finite dimensional Hopf algebra, this cohomology is the same as that of the
Hopf algebra.

One useful property of cohomological support varieties for finite group schemes
is the tensor product property: Friedlander and Pevtsova \cite{FP} used rank
varieties to prove that
the variety of the tensor product of two modules for a finite group scheme is
the intersection of their varieties. 
Rank varieties have not been defined in our general setting, however
we give in Theorem \ref{tp} a weaker tensor
product property in the general setting that suffices for our purposes. That is,
the property holds for any tensor product of an arbitrary module with a special
type of module analogous to those originally constructed by Carlson \cite{Ca} in
the finite group setting, Carlson's modules $L_{\zeta}$. 
This was proved for quantum elementary abelian groups by Pevtsova and the
second author \cite{PW}, and we show here that the proof is valid more generally.
This result is precisely
what is needed to generalize Farnsteiner's results on the representation type of
finite group schemes to finite dimensional Hopf algebras satisfying ({\bf fg}).
We prove that if there exists a module having complexity at least 3 (or equivalently
the support variety of the module has dimension at least 3), then the Hopf algebra
is wild. In fact, we prove a stronger version for blocks in Theorem \ref{mainthm}.

We apply our main theorem in particular to small quantum groups, whose cohomology
was computed originally by Ginzburg and Kumar \cite{GK} and later by Bendel, Nakano,
Parshall, and Pillen \cite{BNPP} under weaker conditions on the parameter. Exploiting
\cite{MPSW} to weaken these conditions even further, we show in Theorems \ref{sqb}
and \ref{sqg} that if the rank of the simple Lie algebra $\g$ is at least 2, then
both the small quantum group $u_q({\g})$ and its Borel-type subalgebra $u_q^+({\g})$
are wild. This proves a conjecture of Cibils for $u_q({\g})$ \cite{C}, who established
the analogous result for $u_q^+(\g)$, in the simply laced case, using completely
different methods.

Throughout the paper, we let $k$ be an algebraically closed field of arbitrary
characteristic. For the applications in Section \ref{QG}, we will take $k=\C$.
We will assume that all modules over Hopf algebras or their blocks are unital
left modules and that all tensor products are over $k$ unless indicated otherwise.


\section{Complexity and varieties}\label{background}

In this section, we first give some basic definitions and results on complexity
and varieties for modules of a finite dimensional Hopf algebra, and then prove
the tensor product property for Carlson's modules $L_{\zeta}$.

Let $V_{\DOT}$ be a graded vector space over $k$ with finite dimensional homogeneous
components. Define the {\em rate of growth} $\gamma(V_{\DOT})$ to be the smallest
non-negative integer $c$ such that there is a number $b$ for which $\dim_k V_n\leq b
n ^{c-1}$ for all positive integers $n$. If no such $c$ exists, we define
$\gamma(V_{\DOT})$ to be $\infty$.

Let $A$ be a finite dimensional Hopf algebra over $k$. We will also denote by $k$
the ground field as an $A$-module via the counit (or augmentation) of $A$.
The {\em complexity} $\cx_A(M)$ of an $A$-module $M$ may be defined in the standard
way (cf.\ \cite[Definition 5.3.4]{Be2}):

Let
$
P_{\DOT}: \ \cdots\rightarrow P_1\rightarrow P_0\rightarrow M\rightarrow 0
$
be a minimal projective resolution of $M$. Then $\cx_A(M):=\gamma(P_{\DOT})$.

We will need the following properties of complexity.

\begin{prop}\label{cx}
Let $A$ be a finite dimensional Hopf algebra over $k$, let $H$ be a Hopf subalgebra
of $A$, and let $M$ be any finite dimensional $A$-module. Then
\begin{itemize}
\item[(1)] $\ \cx_A(M)\le\cx_A(k)$.
\item[(2)] $\ \cx_H(M)\le\cx_A(M)$.
\end{itemize}
\end{prop}

\begin{proof}
(1): Consider a minimal projective resolution $\cdots\rightarrow P_1\rightarrow
P_0\rightarrow k\rightarrow 0$ of the trivial $A$-module $k$. Tensoring this
resolution with $M$ yields a projective resolution $\cdots\rightarrow M\ot P_1
\rightarrow M\ot P_0\rightarrow M\rightarrow 0$ of $M$, since the tensor product
of any module with a projective module is again projective. It is clear that the
rate of growth of this resolution is also $\cx_A(k)$, and because $\cx_A(M)$ is
the rate of growth of a {\it minimal\/} projective resolution of $M$, the desired
inequality follows.

(2): Note that $A$ is a free $H$-module by the Nichols-Zoeller Theorem (see
\cite[Theorem 7]{NZ} or \cite[Theorem 3.1.5]{Mo}), and so every projective
resolution of $A$-modules restricts to a projective resolution of $H$-modules.
\end{proof}

We will use the following result, which can be established using a long exact
cohomology sequence. 

\begin{prop}\label{ses}
Let $A$ be a finite dimensional Hopf algebra over $k$, and let $0\to M_1\to M_2
\to M_3\to 0$ be a short exact sequence of finite dimensional $A$-modules. Then
$$
\cx_A(M_i)\le\max\{\cx_A(M_j),\cx_A(M_k)\}\,,
$$
whenever $\{i,j,k\}=\{1,2,3\}$.
\end{prop}

For any $A$-module $M$, let
$
\coh^{\DOT}(A,M):=\Ext^{\DOT}_A(k,M).
$
The vector space $\coh^{\DOT}(A,k)$ is an associative graded $k$-algebra under cup
product, or equivalently under Yoneda composition (see \cite[Section 3.2]{Be1}). If 
$M$ and $N$ are any two $A$-modules, then $\coh^{\DOT}(A,k)$ acts on $\Ext^{\DOT}_A
(M,N)$ via the cup product, or equivalently as $-\ot N$ followed by Yoneda composition
(see \cite[Proposition 3.2.1]{Be1}).

We use the notational convention that
\begin{eqnarray*}
\coh^{\ev}(A,k):=
\left\{
\begin{array}{cl}
\bigoplus\limits_{n=0}^\infty\coh^n(A,k), & \mbox{if }\cha k = 2,\\ & \\
\bigoplus\limits_{n=0}^\infty\coh^{2n}(A,k), & \mbox{if }\cha k \neq 2\,.
\end{array}
\right.
\end{eqnarray*}
Since $\coh^{\DOT}(A,k)$ is graded commutative (see \cite[Section 5.6]{GK} or
\cite[Section 2.2]{SA}), in either case, $\coh^{\ev}(A,k)$ is a commutative
$k$-algebra.
We need the following assumption for the main results in this paper.
\vspace{.2cm}

\noindent Assumption ({\bf fg}):
\vspace{-.2cm}
\begin{equation*}
\begin{array}{l}
\mbox{{\em Assume }}\coh^{\ev}(A,k) \ \mbox{{\em is finitely generated, and that
for any two finite dimensional}}\\
A\mbox{{\em -modules }}M\mbox{ {\em and} }N, \mbox{ {\em the }} \coh^{\ev}(A,k)
\mbox{{\em -module }}\Ext^{\DOT}_A(M,N) \ \mbox{{\em is  finitely generated.}}
\end{array}
\end{equation*}

\noindent This is known to be the case, for example, if $A$ is finite dimensional
cocommutative \cite{FS}, if $A$ is a small quantum group $u_q(\g)$ under some
restrictions on $q$ \cite{BNPP,GK}, or, more generally, if $A$ is a finite
dimensional pointed Hopf algebra with an abelian group of group-like elements
under some restrictions on the parameters \cite{MPSW} 
or if $A$ is a certain truncated quantized function algebra of a simply-connected
complex semisimple algebraic group for the degree of the parameter being
not too small \cite{G1}.
We note that $\Ext^{\DOT}_A
(M,N)\cong\Ext^{\DOT}_A(k,M^*\ot N)$, an isomorphism of $\coh^{\DOT}(A,k)$-modules.
Therefore the second part of the assumption ({\bf fg}) may be replaced by the
assumption that $\coh^{\DOT}(A,M)$ is a finitely generated $\coh^{\ev}(A,k)$-module
for every finite dimensional $A$-module $M$.

Under the assumption ({\bf fg}), we may define varieties for modules in the usual
way:

Let $M$ and $N$ be $A$-modules. Let $I_A(M,N)$ be the annihilator of the action of
$\coh^{\ev}(A,k)$ on $\Ext^{\DOT}_A(M,N)$, a homogeneous ideal of $\coh^{\ev}(A,k)$,
and let $\V_A(M,N)$ denote the maximal ideal spectrum of the finitely generated
commutative $k$-algebra $\coh^{\ev}(A,k)/I_A(M,N)$. As the ideal $I_A(M,N)$ is
homogeneous, the variety $\V_A(M,N)$ is conical. If $M=N$, we write $I_A(M):=
I_A(M,M)$ and $\V_A(M):=\V_A(M,M)$. The latter is called the {\em support variety}
of $M$.

We will need the following connection between complexity and varieties.

\begin{prop}\label{complexity}
Let $A$ be a finite dimensional Hopf algebra over $k$ satisfying {\rm ({\bf fg})},
and let $M$ be a finite dimensional $A$-module. Then $$\cx_A(M)=\dim\V_A(M)\,.$$
\end{prop}

\begin{proof}
The proof of \cite[Proposition 5.7.2]{Be2} applies in this context. We include
the details for the convenience of the reader. The cup product maps $\coh^{\ev}
(A,k)$ to $\Ext^{\DOT}_A(M,M)$, with kernel $I_A(M)$. By assumption ({\bf fg}),
$\Ext^{\DOT}_A(M,M)$ is a finitely generated module over $\coh^{\ev}(A,k)$, and
so it is a finitely generated module over the quotient $\coh^{\ev}(A,k)/I_A(M)$.
By the definition of $\V_A(M)$ and because Krull dimension and rate of growth of
a finitely generated commutative graded algebra coincide, we have
$$
\dim\V_A(M)=\dim(\coh^{\ev}(A,k)/I_A(M))=\gamma(\coh^{\ev}(A,k)/I_A(M))=\gamma
(\Ext^{\DOT}_A(M,M))\,.
$$
It remains to prove that $\cx_A(M)=\gamma(\Ext^{\DOT}_A(M,M))$.

Let $P_{\DOT}$ be a minimal resolution of $M$. The multiplicity of the projective
cover $P_A(S)$ of a simple $A$-module $S$, as a direct summand of $P_n$, is
$$
\dim_k\Hom_A(P_n,S)=\dim_k\Ext^{n}_A(M,S)\,.
$$
Let $\Irr(A)$ denote a complete set of representatives of isomorphism classes of
simple $A$-modules. It follows that
$$
\dim_k P_n=\sum_{S\in\Irr(A)}\dim_k P_A(S)\cdot\dim_k\Ext^{n}_A(M,S)\,.
$$
This shows that
$$
\gamma(P_{\DOT})\leq\max\{\gamma(\Ext^{\DOT}_A(M,S))\mid S\in\Irr(A)\}\,.
$$
On the other hand, for each simple $A$-module $S$, since $\Ext^{\DOT}(M,S)$ is a
finitely generated $\coh^{\ev}(A,k)$-module, and the action of $\coh^{\ev}(A,k)$
on $\Ext^{\DOT}_A(M,S)$ factors through $\Ext^{\DOT}_A(M,M)$, we have that
$\Ext^{\DOT}_A(M,S)$ is a finitely generated $\Ext^{\DOT}_A(M,M)$-module. This
implies $\gamma(\Ext^{\DOT}_A(M,S))\leq\gamma(\Ext^{\DOT}_A(M,M))$. It is an
immediate consequence of $\dim_k\Ext^n_A(M,M)\leq\dim_k\Hom_k(P_n,M)=(\dim_k M)
(\dim_k P_n)$ for every non-negative integer $n$ that $\gamma(\Ext^{\DOT}_A(M,M))
\leq\gamma(P_{\DOT})$. In conclusion, we obtain
\begin{eqnarray*}
\max\{\gamma(\Ext^{\DOT}_A(M,S))\mid S\in\Irr(A)\}
& \leq & \gamma(\Ext^{\DOT}_A(M,M))\\
& \leq & \gamma(P_{\DOT})\\
& \leq & \max\{\gamma(\Ext^{\DOT}_A(M,S))\mid S\in\Irr(A)\}.
\end{eqnarray*}
Thus all the inequalities above are equalities, and since $\cx_A(M)=\gamma(P_{\DOT})$,
we have shown that $\cx_A(M)=\gamma(\Ext^{\DOT}_A(M,M))$, as desired.
\end{proof}

We will also need the following properties of support varieties for modules. Again
let $\Irr(A)$ denote a complete set of representatives of isomorphism classes of 
simple $A$-modules.

\begin{prop}\label{properties}
Let $A$ be a finite dimensional Hopf algebra over $k$ satisfying {\rm ({\bf fg})},
and let $M$ and $N$ be finite dimensional $A$-modules. Then
\begin{itemize}
\item[(1)] $\V_A(M)=\{0\}$ if and only if $M$ is projective.
\item[(2)] $\V_A(M\oplus N)=\V_A(M)\cup\V_A(N)$.
\item[(3)] $\V_A(M,N)\subseteq\V_A(M)\cap\V_A(N)$.
\item[(4)] $\V_A(M)=\cup_{S\in\Irr(A)}\V_A(M,S)=\cup_{S\in\Irr(A)}\V_A(S,M)$.
\item[(5)] $\V_A(M\ot N)\subseteq\V_A(M)\cap\V_A(N)$.
\end{itemize}
\end{prop}

\begin{proof}
The proofs of (1)--(4) are exactly analogous to the proofs in \cite[Section 5.7]{Be2}.

The proof of (5) is standard, however we provide the details for the convenience
of the reader: The action of $\coh^{\ev}(A,k)$ on $\Ext^{\DOT}_A(M\ot N, M\ot N)$
factors through its action on $\Ext^{\DOT}_A(M,M)$, as we may first apply $-\ot M$,
then $-\ot N$, and finally Yoneda composition. Thus $I_A(M,M)\subseteq I_A(M\ot N,M
\ot N)$, which implies that $\V_A(M\ot N)\subseteq\V_A(M)$. To show that $\V_A(M\ot
N)\subseteq\V_A(N)$, note that $\Ext^{\DOT}_A(M\ot N,M\ot N)\cong\Ext^{\DOT}_A(N,
M^*\ot M\ot N)$, so $\V_A(M\ot N)=\V_A(N,M^*\ot M\ot N)$, which is contained in
$\V_A(N)$ by (3).
\end{proof}

Ostrik conjectured in \cite[Remark 3.7(ii)]{O} that the containment in (5) is an
equality. In Theorem \ref{tp} below, we will prove equality in a special case.

Let $\Omega_A(M)$ denote the kernel of an epimorphism from 
the projective cover of $M$ to $M$, let $n>0$, and
let $0\neq\zeta\in\coh^{n}(A,k)\cong\Hom_A(\Omega_A^{n}(k),k)$. We will define a
corresponding $A$-module $L_{\zeta}$, following the original construction by Carlson
\cite{Ca} for finite groups: Let $L_\zeta$ be the kernel of a representative map
$
\widehat\zeta:\Omega_A^n(k)\rightarrow k\,.
$
So $L_{\zeta}$ is defined by the short exact sequence
\begin{equation*}
0\longrightarrow L_\zeta\longrightarrow\Omega_A^n(k)\stackrel{\hat{\zeta}}
\longrightarrow k\longrightarrow 0\,.
\end{equation*}

For each $\zeta\in\coh^{\ev}(A,k)$, denote by $\langle\zeta\rangle$ the ideal of
$\coh^{\ev}(A,k)$ generated by $\zeta$. For each ideal $I$ of $\coh^{\ev}(A,k)$,
denote by $Z(I)$ its zero set, that is the set of all maximal ideals of $\coh^{\ev}
(A,k)$ containing $I$.

The following theorem is a special case of the conjectured equality $\V_A(M\ot N)
=\V_A(M)\cap\V_A(N)$.

\begin{thm}\label{tp}
Let $A$ be a finite dimensional Hopf algebra over $k$ satisfying {\rm ({\bf fg})},
let $M$ be a finite dimensional $A$-module, and let $\zeta$ be a 
non-zero homogeneous element
of positive degree in $\coh^{\ev}(A,k)$. Then
$$
\V_A(M\ot L_{\zeta})=\V_A(M)\cap Z(\langle\zeta\rangle)\,.
$$
In particular, $\V_A(L_{\zeta})=Z(\langle\zeta\rangle)$.
\end{thm}

\begin{proof}
The proof is essentially that of \cite[Proposition 3]{PW} where it is stated only
in the case that $A$ is a quantum elementary abelian group. This result is in fact
valid under our more general assumptions. 

Let $N$ and $N'$ be finite dimensional $A$-modules, and let $\m$ be a maximal ideal
in $\coh^{\ev}(A,k)$.
Since $\Ext^{\DOT}_A(N,N')$ is finitely generated over $\coh^{\ev}(A,k)$ by assumption
({\bf fg}), we have that $\m\in\V_A(N,N')$ if and only if $I_A(N,N')\subseteq\m$ if and
only if $\Ext^{\DOT}_A(N,N')_\m\neq 0$.

We will first show that $\V_A(M)\cap Z(\langle\zeta\rangle)\subseteq\V_A(M\ot L_{\zeta})$.
By applying Proposition~\ref{properties}(4) to $M$ and to $M\ot L_{\zeta}$, we have
$$
\V_A(M)=\cup_{S\in\Irr(A)}\V_A(M,S)\mbox{ and }
\V_A(M\ot L_{\zeta})=\cup_{S\in\Irr(A)}\V_A(M\ot L_{\zeta},S)\,.
$$
Thus it suffices to show that
$$
\V_A(M,S)\cap Z(\langle\zeta\rangle)\subseteq\V_A(M\ot L_{\zeta},S)
$$
for every simple $A$-module $S$. Let $\m$ be a maximal ideal in $\V_A(M,S)\cap Z(\langle
\zeta\rangle)$. Then $I_A(M,S)\subseteq\m$ and $\langle\zeta\rangle\subseteq\m$, and
therefore $\m$ contains the ideal generated by $I_A(M,S)$ and $\zeta$. We must show that
$\m\in\V_A(M\ot L_{\zeta},S)$, that is $I_A(M\ot L_{\zeta},S)\subseteq\m$.

Suppose $I_A(M\ot L_{\zeta},S)\not\subseteq\m$. Then as noted above, $\Ext^{\DOT}_A
(M\ot L_{\zeta},S)_\m=0$. Apply $M\ot-$ and $\Ext^{\DOT}_A(-,S)$ to the exact sequence
$$
0\rightarrow L_{\zeta}\rightarrow\Omega_A^n(k)\stackrel{\hat{\zeta}}{\longrightarrow}k
\rightarrow 0\,.
$$
Now $\Omega_A^n(M)$ is isomorphic to $M\ot\Omega_A^n(k)$ up to projective direct
summands, and so $\Ext^i_A(M\ot\Omega_A^n(k),S)\cong \Ext^{i+n}_A(M,S)$. Thus we
obtain a long exact sequence
$$
\cdots\rightarrow\Ext^i_A(M,S)\stackrel{\tilde{\zeta}}{\rightarrow}\Ext^{i+n}_A(M,S)
\stackrel{\eta}{\rightarrow}\Ext^i_A(M\ot L_{\zeta},S)\stackrel{\delta}{\rightarrow}
\Ext^{i+1}_A(M,S)\rightarrow\cdots
$$
where the map $\tilde{\zeta}:\Ext^i_A(M,S)\rightarrow\Ext^{i+n}_A(M,S)$ is just the
action of $\zeta\in\coh^n(A,k)$ on $\Ext_A^i(M,S)$. The maps in the sequence are
$\coh^{\ev}(A,k)$-module homomorphisms. Let $z\in\Ext^{i+n}_A(M,S)$, and let $a$ be
any homogeneous element of $I_A(M\ot L_{\zeta},S)$ that is not in $\m$. So we have
$\eta(az)=a\eta(z)=0$. Considering the above long exact sequence, since $\eta(az)=0$,
we have $az=\zeta y$ for some $y\in \Ext_A^{i+\deg(a)}(M,S)$, so that $z=\zeta a^{-1}y
\in\zeta\Ext^{\DOT}_A(M,S)_\m$. Therefore
$$
\Ext^j_A(M,S)_\m=\zeta \Ext^j_A(M,S)_\m
$$
for all $j>n$. We will show that this also holds for $j\leq n$. Assume $z\in\Ext^j_A(M,S)$
for $j\leq n$. Let $b$ be a homogeneous element of positive degree in $\coh^{\ev}(A,k)$
that is not in $\m$. Multiply $z$ by a large enough power of $b$ so that $\deg(b^m z)>n$.
Then $b^mz\in \Ext^{\DOT}_A(M,S)_\m=\zeta \Ext^{\DOT}_A(M,S)_\m$ as above. Now $b$ is
invertible in $\Ext^{\DOT}_A(M,S)_\m$, so we obtain $z\in\zeta\Ext^{\DOT}_A(M,S)_\m$.

Since $\zeta\in\m$ and $\Ext^{\DOT}_A(M,S)$ is finitely generated over $\coh^{\ev}(A,k)$
by assumption ({\bf fg}), Nakayama's Lemma applied to the local ring $\coh^{\ev}(A,k)_\m$
implies that $\Ext^{\DOT}_A(M,S)_\m=0$. This contradicts the assumption $I_A(M,S)\subseteq
\m$. Therefore $I_A(M\ot L_{\zeta},S)\subseteq\m$, and so $\V_A(M,S)\cap Z(\langle\zeta
\rangle)\subseteq \V_A(M\ot L_{\zeta},S)$.

To prove the opposite inclusion $\V_A(M\ot L_{\zeta})\subseteq\V_A(M)\cap Z(\langle\zeta
\rangle)$, by Proposition~\ref{properties}(5) it suffices to show that $\V_A(L_{\zeta})
\subseteq Z(\langle\zeta\rangle)$. Applying Proposition \ref{properties}(4) again, it is
enough to show that $\V_A(L_{\zeta},S)\subseteq Z(\langle\zeta\rangle)$ for every simple
$A$-module $S$. Thus we need to show that if $\m$ is a maximal ideal of $\coh^{\ev}(A,k)$
for which $\Ext^{\DOT}_A(L_{\zeta},S)_\m\neq 0$, then $\zeta\in\m$. Assume to the contrary
that $\zeta\not\in\m$. Then multiplication by $\zeta$ induces an isomorphism on
$\Ext^{\DOT}_A(k,S)_\m$, since $\zeta$ is invertible in $\coh^{\ev}(A,k)_\m$. As localization
is exact, the existence of the short exact sequence defining $L_{\zeta}$ implies that
$\Ext^{\DOT}_A(L_{\zeta},S)_\m$ is the kernel of the isomorphism $\zeta:\Ext^{\DOT}_A
(k,S)_\m\rightarrow\Ext^{\DOT+n}_A(k,S)_\m$ and therefore $\Ext^{\DOT}_A(L_{\zeta},S)_\m=0$.

Finally, by setting $M=k$ we obtain the second statement.
\end{proof}

As a consequence of the theorem, we may find modules for blocks having prescribed
varieties:

\begin{cor}\label{block}
Let $A$ be a finite dimensional Hopf algebra over $k$ satisfying {\rm ({\bf fg})}. Let
$B$ be a block of $A$, let $M$ be a finite dimensional $B$-module, and let $d$ be a
positive integer. Then for any non-zero $\zeta\in\coh^{2d}(A,k)$, there exists a $B$-module
$N_{\zeta}$ such that $\V_A(N_{\zeta})=Z(\langle\zeta\rangle)\cap\V_A(M)$ and $\dim_k
N_{\zeta}\leq(\dim_k M)(\dim_k\Omega^{2d}_A(k))$.
\end{cor}

\begin{proof}
The proof is exactly the same as \cite[Proposition 2.2]{F}, using Theorem \ref{tp} and
Proposition \ref{properties}(1),(2) in the last part of the argument. The module $N_{\zeta}$
is defined to be $e\cdot(M\ot L_{\zeta})$, where $e$ is the primitive central idempotent
associated to $B$. The crucial observation in the proof is that $(1-e)\cdot (M\ot
L_{\zeta})$ is projective since $\Omega^{2d}_A(M)$ and $M\ot \Omega^{2d}_A(k)$ differ
by a projective module.
\end{proof}

\begin{remark}\label{any}
A further consequence of the theorem is that any conical subvariety of $\V_A(k)$ can
be realized as the support variety of some module: If $I=\langle\zeta_1,\ldots,\zeta_t
\rangle$ is any homogeneous ideal in $\coh^{\ev}(A,k)$, we may successively apply both
parts of Theorem \ref{tp} to obtain the $A$-module $M=L_{\zeta_1}\ot\cdots\ot L_{\zeta_t}$
with the property that $\V_A(M)=Z(I)$. Similarly, we may use Corollary \ref{block} to
obtain a generalization of \cite[Corollary 2.3]{F}, namely that any conical subvariety
of the variety $\V_B:=\cup_{S\in\Irr(B)}\V_A(S)$ of a block $B$ can be realized as the
support variety of some $B$-module. Avramov and Iyengar \cite[Existence Theorem 5.4]{AI}
proved a more general realizability result that does not use tensor products. 
\end{remark}


\section{Complexity and representation type}

Theorem \ref{mainthm} below generalizes \cite[Theorem 3.1]{F}. We show that Farnsteiner's
proof works in this more general context, as an application of the general theory developed
in Section~\ref{background}. For a similar result in the setting of self-injective algebras,
see \cite[Corollary 4.2]{BS}.
Here we develop and apply the theory for $\coh^{\DOT}(A,k)$ directly,
in the spirit of prior work on finite group schemes.

Recall that finite dimensional associative algebras over an algebraically closed
field can be divided into three classes (see \cite[Corollary C]{CB} or \cite[Theorem
4.4.2]{Be1}): An algebra $A$ is {\em representation-finite\/} if there are only
finitely many isomorphism classes of finite dimensional indecomposable $A$-modules.
An algebra $A$ is called {\em tame\/} if it is not representation-finite and if the
isomorphism classes of indecomposable $A$-modules in any fixed dimension are almost
all contained in a finite number of one-parameter families. An algebra $A$ is said
to be {\em wild\/} if the category of finite dimensional $A$-modules contains the
category of finite dimensional modules over the free associative algebra in two
indeterminates. (For more precise definitions we refer the reader to \cite[Definition
4.4.1]{Be1}.) Note that the classification of indecomposable objects (up to isomorphism)
of the latter category is a well-known unsolvable problem and so one is only able
to classify the finite dimensional indecomposable modules of representation-finite
or tame algebras.

\begin{thm}\label{mainthm}
Let $A$ be a finite dimensional Hopf algebra over $k$ for which assumption
{\rm ({\bf fg})} holds, and let $B$ be a block of $A$. If there is a finite
dimensional $B$-module $M$ such that $\cx_A(M)\ge 3$, then $B$ is wild.
\end{thm}

\begin{proof}
Let $\Irr(B)$ denote a complete set of representatives for the isomorphism classes
of simple $B$-modules. Let $M_B:=\oplus_{S\in \Irr(B)} S$, let $\V_B:= \V_A(M_B)$
denote the support variety of $M_B$, and set $n:=\dim\V_B$. It follows from 
Proposition \ref{ses} that $n\ge\cx_A(M)\ge 3$. By \cite[Lemma 1.1]{F}, there are
non-zero elements $\zeta_s$ ($s\in k$) in $\coh^{2d}(A,k)$ for some $d>0$ such that
\begin{itemize}
\item[(i)]  $\dim Z(\langle\zeta_s\rangle)\cap \V_B=n-1$ for all $s\in k$,
\item[(ii)] $\dim Z(\langle\zeta_s\rangle)\cap Z(\langle\zeta_t\rangle)\cap\V_B=n-2$
            for $s\neq t\in k$.
\end{itemize}
By applying Corollary \ref{block} to $M_B$, there are $B$-modules $N_{\zeta_s}$ such
that
$$
\V_A(N_{\zeta_s})=Z(\langle\zeta_s\rangle)\cap\V_B
\ \ \mbox{ and } \ \
\dim_k N_{\zeta_s}\leq (\dim_k M_B)(\dim_k \Omega_A^{2d}(k))
$$
for all $s\in k$. Decompose $N_{\zeta_s}$ into a direct sum of indecomposable $B$-modules.
By Proposition \ref{properties}(2) and (i) above, there is an indecomposable direct summand
$X_s$ of $N_{\zeta_s}$ such that
$$
\V_A(X_s)\subseteq\V_A(N_{\zeta_s})=Z(\langle\zeta_s\rangle)\cap\V_B
$$
and $\dim\V_A(X_s)=n-1$. We will show that $\V_A(X_s)\neq\V_A(X_t)$ when $s\neq t$:
Note that if $\V_A(X_s)=\V_A(X_t)$ then
$$
\V_A(X_s)\subseteq Z(\langle\zeta_s\rangle)\cap Z(\langle\zeta_t\rangle)\cap\V_B\,,
$$
and the latter variety has dimension $n-2$ if $s\neq t$. Therefore the varieties $\V_A(X_s)$
are distinct for different values of $s$, implying that the indecomposable $B$-modules $X_s$
are pairwise non-isomorphic. The dimensions of the modules $X_s$ are all bounded by $(\dim_k
M_B)(\dim_k\Omega_A^{2d}(k))$, since $X_s$ is a direct summand of $N_{\zeta_s}$. Consequently,
there are infinitely many non-isomorphic indecomposable $B$-modules $X_s$ of some fixed
dimension. By Proposition \ref{complexity}, $\cx_A(X_s)=\dim\V_A(X_s)=n-1\geq 2$.

Suppose now that $B$ is not wild. By \cite[Corollary C]{CB}, $B$ is tame or represen\-tation-finite
and it follows from \cite[Theorem D]{CB} that only finitely many indecomposable $B$-modules
of any dimension (up to isomorphism) are not isomorphic to their Auslander-Reiten translates.
Since the Auslander-Reiten translation $\tau_A$ is the same as $\nu_A\circ\Omega_A^2=\Omega_A^2
\circ\nu_A$ where $\nu_A$ denotes the Nakayama automorphism of $A$ (see \cite[Proposition
IV.3.7(a)]{ARS}) and the Nakayama automorphism of any finite dimensional Hopf algebra
has finite order (see \cite[Lemma~1.5]{FMS}), it is clear that any module isomorphic to its
Auslander-Reiten translate is periodic and thus has complexity 1. Hence in any dimension
there are only finitely many isomorphism classes of indecomposable $B$-modules with complexity
not equal to 1. This cannot be the case, as we have shown that for some dimension there
are infinitely many non-isomorphic indecomposable $B$-modules of complexity greater than~1. 
Therefore $B$ is wild.
\end{proof}

As an immediate consequence of Theorem \ref{mainthm} and the Trichotomy Theorem
\cite[Corollary C]{CB} one obtains as in \cite[Corollary 3.2]{F}:

\begin{cor}\label{maincor}
Let $A$ be a finite dimensional Hopf algebra over $k$ for which assumption {\rm ({\bf fg})}
holds. If $B$ is a tame block of $A$, then $\cx_A(M)\le 2$ for every finite dimensional
$B$-module $M$.
\end{cor}


\section{Small quantum groups}\label{QG}

Let $\g$ be a finite dimensional complex simple Lie algebra, $\Phi$ its root system, $h$
its Coxeter number, and $r:=\rank(\Phi)$ its rank. Let $\ell>1$ be an odd integer and
assume that $\ell$ is not divisible by $3$ if $\Phi$ is of type G$_2$. Let $q$ be a
primitive $\ell$-th root of unity in $\C$. Let $U_q(\g)$ denote Lusztig's quantum group
\cite{L1} and let $u_q(\g)$ denote the small quantum group \cite{L2}.

Now fix a set of simple roots, and let $\Phi^+$ and $\Phi^-$ denote the corresponding sets
of positive and negative roots, respectively. Then $\g$ has a standard Borel subalgebra
corresponding to $\Phi^+$ and an opposite standard Borel subalgebra corresponding to
$\Phi^-$. We will also use the notation $u_q^+(\g)$ to denote the Hopf subalgebra of
$u_q({\g})$ corresponding to the standard Borel subalgebra of $\g$ and $u_q^-(\g)$ to
denote the Hopf subalgebra of $u_q(\g)$ corresponding to the opposite standard Borel
subalgebra of $\g$. Note that $u_q^-(\g)\cong u_{q^{-1}}^+(\g)$ (cf.\ \cite[(1.2.10)]{CK}).

We will need some knowledge about the complexity of the trivial $u_q^+(\g)$-module
$\C$ (which according to Proposition \ref{cx}(1) is the module of largest complexity).
Let $\rho$ be half the sum of the positive roots, let $\langle-,-\rangle$ denote the
positive definite symmetric bilinear form on the real vector space spanned by $\Phi$
such that $\langle\alpha,\alpha\rangle=2$ if $\alpha$ is a short root of $\Phi$, and
set
$$
\Phi_0^+:=\{\alpha\in\Phi^+\mid\langle\rho,\alpha^{\vee}\rangle\in\ell\Z\}\,.
$$
Let $U_q^0(\g)$ denote the subalgebra of $U_q(\g)$ generated by the group-like elements
$K_i^{\pm 1}$ ($1\le i\le r$) and the elements $\left[K_i;N\atop n\right]$ ($1\le i\le
r$; $N\in\Z$, $n\in\Z_{\ge 0}$) (see \cite[(4.1)(b)]{L1}). Since Weyl's character formula
holds for the quantized induced modules $H_q^0(\lambda)$ (see \cite[Corollary 5.12]{APW})
and the rest of the argument in \cite[\S2.2--\S2.5]{UGA} depends only on the characters
of the involved modules, one can prove the following result by replacing the category of
$B_1T$-modules by the category of $u_q^+(\g)U_q^0(\g)$-modules, the induced module $H^0
(\lambda)$ by the quantized induced module $H_q^0(\lambda)$, $p$ by $\ell$, and then
specializing to $\lambda=0$ (see \cite[(2.5.2)]{UGA} for the analogue for the first
Frobenius kernel $B_1$ of the Borel subgroup $B$):

\begin{lem}\label{lambda}
Let $q$ be a primitive $\ell$-th root of unity and assume that $\ell>1$ is an odd integer
not divisible by $3$ if $\Phi$ is of type $G_2$. Then $\cx_{u_q^+(\g)}(\C)\ge\vert\Phi^+
\vert-\vert\Phi_0^+\vert$.
\end{lem}

The following result, a consequence of \cite[Theorem 5.6]{D} or \cite[Theorem 6.3]{MPSW},
ensures that assumption ({\bf fg}) holds for $A=u_q^+(\g)$. In these papers it is shown
that the full cohomology ring is finitely generated. It follows that the even cohomology
ring is finitely generated: Without loss of generality, the generators of the full cohomology
ring are homogeneous. Take the set consisting of all generators in even degree and all
products of pairs of generators in odd degree. Since odd degree elements are nilpotent,
this set generates the cohomology in even degrees. A similar argument applies to modules
over the cohomology ring.

\begin{thm}\label{uqb}
Let $q$ be a primitive $\ell$-th root of unity and assume that $\ell$ is an odd integer not
divisible by $3$ if $\Phi$ is of type $G_2$. Then the even cohomology ring $\coh^{\ev}(u_q^+
(\g),\C)$ is finitely generated. Moreover, if $M$ is a finite dimensional $u_q^+(\g)$-module,
then $\coh^{\DOT}(u_q^+(\g),M)$ is finitely generated as an $\coh^{\ev}(u_q^+(\g),\C)$-module.
\end{thm}

The following theorem was proved by Cibils \cite[Propositions 3.1 and 3.3]{C} in the simply
laced case for $\ell\geq 5$ by purely representation-theoretic methods using quivers and
bimodule complements,
 and, in general, for $\ell>h$
      by Gordon \cite[Theorem 7.1(a)(i)]{G2}
     using geometric methods similar to ours (see also \cite[Corollary~4.6(ii)]{BG}
     for replacing the restriction $\ell>h$ by $\ell$ good).

\begin{thm}\label{sqb}
Let $r\ge 2$ and let $q$ be a primitive $\ell$-th root of unity. Assume that $\ell>1$ is
an odd integer not divisible by $3$ if $\Phi$ is of type $G_2$. Then $u_q^+(\g)$ is wild.
\end{thm}

Note that $u_q^+(\g)$ has only one block. If $r=1$, then $u_q^+(\g)$ is known to be
representation-finite \cite[Proposition 3.3]{C}.

\begin{proof}
First suppose $\ell\geq h$. Then $\Phi_0^+=\emptyset$ and consequently Lemma \ref{lambda}
implies that $\cx_{u_q^+(\g)}(\C)\geq\vert\Phi^+\vert\geq 3$ since the rank $r$ of $\Phi$
is at least $2$ by hypothesis. In particular, this argument applies to $A_2$ for any $\ell$,
since in this case $h=3$. By Theorem \ref{mainthm}, $u_q^+(\g)$ is wild.

Now suppose $\ell<h$ and that $\g$ is neither of type $B_2$ nor of type $G_2$. Then
$u_q^+(\g)$ contains $u_q^+(\mathfrak{sl}_3)$ as a Hopf subalgebra. Applying Proposition
\ref{cx}(2), because $\ell\geq 3$ by hypothesis, it follows that $\cx_{u_q^+(\g)}(\C)\ge
\cx_{u_q^+(\mathfrak{sl}_3)}(\C)\ge 3$, and as before, $u_q^+(\g)$ is wild by Theorem
\ref{mainthm}.

It remains to consider the cases when $\g$ is of type $B_2$ or of type $G_2$, and $\ell
<h$. If $\g$ is of type $B_2$, then $h=4$, and by hypothesis $\ell=3$, which yields
$\vert\Phi_0^+\vert=1$. It follows from Lemma \ref{lambda} that $\cx_{u_q^+(\g)}(\C)\ge 3$,
implying again that $u_q^+(\g)$ is wild by Theorem \ref{mainthm}. If $\g$ is of type $G_2$,
then $h=6$ and by hypothesis $\ell=5$, which yields $\vert\Phi_0^+\vert=1$. It follows as
before that $\cx_{u_q^+(\g)}(\C)\ge 5$, implying again that $u_q^+(\g)$ is wild by Theorem
\ref{mainthm}.
\end{proof}

\begin{remark}
By carrying out the first part of the proof only for type $A_2$ and instead dealing with
$B_2$ and $G_2$ also for $\ell\ge h$ one indeed needs to use Lemma \ref{lambda} only in
the rank two case.
\end{remark}

The following result is a consequence of \cite[Corollary 6.5]{MPSW}, and it ensures that
assumption ({\bf fg}) holds for $A=u_q(\g)$.

\begin{thm}\label{uqg}
Let $q$ be a primitive $\ell$-th root of unity and assume that $\ell>1$ is an odd integer
not divisible by $3$ if $\Phi$ is of type G$_2$. Then the even cohomology ring $\coh^{\ev}
(u_q(\g),\C)$ is finitely generated. Moreover, if $M$ is a finite dimensional $u_q(\g)$-module,
then $\coh^{\DOT}(u_q(\g),M)$ is finitely generated as an $\coh^{\ev}(u_q(\g),\C)$-module.
\end{thm}

Note that under slightly stronger conditions on $\ell$ the above theorem is also a consequence
of the explicit computation of the cohomology ring $\coh^{\DOT}(u_q(\g),\C)$ which is due to
Ginzburg and Kumar \cite[Theorem 3]{GK} for $\ell>h$, and to Bendel, Nakano, Parshall, and
Pillen \cite[Theorem 1.3.4]{BNPP} for $\ell\leq h$ such that $\ell$ is neither a bad prime
for $\Phi$ nor $\ell=3$ for $\Phi$ of types $B$ and $C$.

The {\em principal block} of a Hopf algebra is defined to be the block corresponding
to the one-dimensional trivial module. The following theorem implies a conjecture of
Cibils \cite[p.\ 542]{C}.

\begin{thm}\label{sqg}
Let $r\ge 2$ and let $q$ be a primitive $\ell$-th root of unity. Assume that $\ell>1$
is an odd integer not divisible by $3$ if $\Phi$ is of type $G_2$. Then the principal
block of $u_q(\g)$ is wild.
\end{thm}

If $r=1$, then $u_q(\g)$ is known to be tame \cite{S,X}.

\begin{proof}
It follows from the proof of Theorem \ref{sqb} that $\cx_{u_q^+(\g)}(\C)\ge 3$ if $r\ge 2$.
Since $u_q^+(\g)$ is a Hopf subalgebra of $u_q(\g)$, we can apply Proposition \ref{cx}(2)
to obtain $\cx_{u_q(\g)}(\C)\ge\cx_{u_q^+(\g)}(\C)\ge 3$. Then Theorem \ref{mainthm}
implies that the principal block of $u_q(\g)$ is wild.
\end{proof}

\begin{remark}
It is also possible to give a direct proof of Theorem \ref{sqg} along the lines of the
proof of Theorem \ref{sqb} by replacing $\Phi_0^+$ by $\Phi_0:=\{\alpha\in\Phi\mid
\langle\rho,\alpha^{\vee}\rangle\in\ell\Z\}$ and then using \cite[Theorem 8.2.1(a)]{BNPP}
instead of Lemma \ref{lambda}.
\end{remark}

  Note that Gordon also proved the wildness of certain truncated quantized
     function algebras of simply-connected connected complex semisimple algebraic
     groups at roots of unity of odd degree (not divisible by $3$ if the group has a
     component of type $G_2$) \cite[Remark (ii) after Theorem 7.1]{G2}.

We conclude by mentioning that Cibils' proof also works for even $\ell$ to which we hope
to return in a future paper. Note also that Cibils' proof uses the subalgebra of $u_q^+
(\g)$ corresponding to the largest nilpotent ideal of the standard Borel subalgebra of
$\g$ (which is not a Hopf subalgebra) whereas our proof does not need this.




\begin{thebibliography}{99}

\bibitem{APW} H.\ H.\ Andersen, P.\ Polo, and K. Wen, ``Representations of quantum
algebras,'' Invent.\ Math.\ 104 (1991), no.\ 1, 1--59; Addendum, Invent.\ Math.\ 120
(1995), no.\ 2, 409--410.

\bibitem{ARS} M.\ Auslander, I.\ Reiten, and S. O. Smal\o{}, {\em Representation
Theory of Artin Algebras} (Corrected reprint of the 1995 original), Cambridge
Studies in Advanced Mathematics, vol.\ 36, Cambridge University Press, Cambridge,
1997.

\bibitem{AI} L.\ L.\ Avramov and S.\ B.\ Iyengar, ``Constructing modules with
prescribed cohomological support,'' Illinois J.\ Math. 51 (2007), no.\ 1, 1--20.

\bibitem{BNPP} C.\ Bendel, D.\ Nakano, B.\ Parshall, and C.\ Pillen,
``Cohomology for quantum groups via the geometry of the nullcone,'' Preprint, 2007.

\bibitem{Be1} D.\ J.\ Benson, {\em Representations and Cohomology I: Basic
Representation Theory of Finite Groups and Associative Algebras\/}, Cambridge Studies
in Advanced Mathematics, vol.\ 30, Cambridge University Press, Cambridge, 1991.

\bibitem{Be2} D.\ J.\ Benson, {\em Representations and Cohomology II: Cohomology of
Groups and Modules\/}, Cambridge Studies in Advanced Mathematics, vol.\ 31, Cambridge
University Press, Cambridge, 1991.

\bibitem{BS} P.\ A.\ Bergh and \O.\ Solberg, ``Relative support varieties,'' to appear
in Q.\ J.\ Math.

\bibitem{BG}   K.\ A.\ Brown and I.\ G.\ Gordon,  ``The ramifications of the centres:
quantised function algebras at roots of unity,''  Proc.\ London Math.\ Soc.\ (3)
84 (2002), no.\ 1, 147--178.

\bibitem{Ca} J.\ F.\ Carlson, ``The variety of an indecomposable module is connected,''
Invent.\ Math.\ 77 (1984), no.\ 2, 291--299.

\bibitem{C} C.\ Cibils, ``Half-quantum groups at roots of unity, path algebras, and
representation type,'' Internat.\ Math.\ Res.\ Notices 1997, no.\ 12, 541--553.

\bibitem{CB} W.\ Crawley-Boevey, ``On tame algebras and bocses,'' Proc.\ London Math.\
Soc.\ (3) 56 (1988), no.\ 3, 451--483.

\bibitem{CK} C.\ De Concini and V.\ G.\ Kac, Representations of quantum groups at
roots of $1$, in: {\em Operator Algebras, Unitary Representations, Enveloping Algebras,
and Invariant Theory, Paris, 1989\/} (eds.\ A. Connes, M.\ Duflo, A. Joseph, and R.\
Rentschler), Progr.\ Math., vol.\ 92, Birkh\"auser, Boston, MA, 1990, pp.\ 471--506.

\bibitem{D} C.\ M.\ Drupieski, {\em Cohomology of Frobenius-Lusztig Kernels of Quantized
Enveloping Algebras\/}, Ph.D. thesis, University of Virginia, 2009.

\bibitem{EO} P.\ Etingof and V.\ Ostrik, ``Finite tensor categories,'' Mosc.\ Math.\ J.\
4 (2004), no.\ 3, 627--654.

\bibitem{F} R.\ Farnsteiner, ``Tameness and complexity of finite group schemes,'' Bull.\
London Math.\ Soc.\ 39 (2007), no.\ 1, 63--70.

\bibitem{FMS} D.\ Fischman, S.\ Montgomery, and H.-J.\ Schneider, ``Frobenius extensions
of subalgebras of Hopf algebras,'' Trans.\ Amer.\ Math.\ Soc.\ 349 (1997), no.\ 12,
4857--4895.

\bibitem{FP} E.\ M.\ Friedlander and J.\ Pevtsova, ``Representation-theoretic support
spaces for finite group schemes,'' Amer.\ J.\ Math.\ 127 (2005), no.\ 2, 379--420;
Erratum, Amer.\ J.\ Math.\ 128 (2006), no.\ 4, 1067--1068.

\bibitem{FS} E.\ M.\ Friedlander and A.\ Suslin, ``Cohomology of finite group schemes
over a field,'' Invent.\ Math.\ 127 (1997), no.\ 2, 209--270.

\bibitem{GK} V.\ Ginzburg and S.\ Kumar, ``Cohomology of quantum groups at roots of
unity,'' Duke Math.\ J.\ 69 (1993), no.\ 1, 179--198.

\bibitem{G1}  I.\ G.\ Gordon, ``Cohomology of quantized function algebras at
roots of unity,'' Proc.\ London Math.\ Soc.\ (3) 80 (2000), no.\ 2, 337--359.

\bibitem{G2}  I.\ G.\ Gordon,  `` Complexity of representations of quantised
function algebras and representation type,''  J.\ Algebra 233 (2000), no.\ 2,
437--482.

\bibitem{L1} G.\ Lusztig, ``Modular representations and quantum groups,'' in:
{\em Classical Groups and Related Topics, Beijing, 1987\/} (eds.\ A.\ J.\ Hahn, D.\
G.\ James, and Z.-X.\ Wan), Contemp.\ Math., vol.\ 82, Amer.\ Math.\ Soc., Providence,
RI, 1989, pp.\ 59--77.

\bibitem{L2} G.\ Lusztig, ``Finite-dimensional Hopf algebras arising from quantized
universal enveloping algebras,'' J.\ Amer.\ Math.\ Soc.\ 3 (1990), no.\ 1, 257--296.

\bibitem{MPSW} M.\ Mastnak, J.\ Pevtsova, P.\ Schauenburg, and S.\ Witherspoon,
``Cohomology of finite dimensional pointed Hopf algebras,'' to appear in
Proc.\ London Math.\ Soc.

\bibitem{Mo} S.\ Montgomery, {\em Hopf Algebras and Their Actions on Rings\/},
CBMS Regional Conference Series in Mathematics, vol.\ 82, Amer.\ Math.\ Soc.,
Providence, RI, 1993.

\bibitem{NZ} W.\ D.\ Nichols and M.\ B.\ Zoeller, ``A Hopf algebra freeness theorem,''
Amer.\ J.\ Math.\ 111 (1989), no.\ 2, 381--385.

\bibitem{O} V.\ Ostrik, ``Support varieties for quantum groups,'' Funct.\ Anal.\
Appl.\ 32 (1998), no.\ 4, 237--246.

\bibitem{PW} J.\ Pevtsova and S.\ Witherspoon, ``Varieties for modules of quantum
elementary abelian groups,'' to appear in Algebras Rep.\ Theory.

\bibitem{R} J.\ Rickard, ``The representation type of self-injective algebras,''
Bull.\ London Math.\ Soc.\ 22 (1990), no.\ 6, 540--546.

\bibitem{SA} M.\ Suarez-Alvarez, ``The Hilton-Eckmann argument for the anti-commutativity
of cup products,'' Proc.\ Amer.\ Math.\ Soc.\ 132 (2004), no.\ 8, 2241--2246.

\bibitem{S} R.\ Suter, ``Modules over $\mathfrak{U}_q(\mathfrak{sl}_2)$,''
Comm.\ Math.\ Phys.\ 163 (1994), no.\ 2, 359--393.

\bibitem{UGA} University of Georgia VIGRE Algebra Group, ``Support varieties for Weyl
modules over bad primes,'' J. Algebra 312 (2007), no.\ 2, 602--633.

\bibitem{X} J.\ Xiao, ``Finite-dimensional representations of $U_t(\mathrm{sl}(2))$ at
roots of unity,'' Canad.\ J.\ Math.\ 49 (1997), no.\ 4, 772--787.

\end{thebibliography}
\end{document}